\documentclass[11pt, psamsfonts]{amsart}

\usepackage{amssymb,amsfonts}
\usepackage[all,arc]{xy}
\usepackage{enumerate}
\usepackage{mathrsfs}
\usepackage{thmtools}
\usepackage{thm-restate}
\usepackage{datetime}
\usepackage{hyperref}
\usepackage{graphicx}
\usepackage{cleveref}

\calclayout

\newcommand\blfootnote[1]{%
  \begingroup
  \renewcommand\thefootnote{}\footnote{#1}%
  \addtocounter{footnote}{-1}%
  \endgroup
}

\newtheorem{thm}{Theorem}[section]
\newtheorem{cor}[thm]{Corollary}
\newtheorem{prop}[thm]{Proposition}
\newtheorem{lem}[thm]{Lemma}
\newtheorem{conj}[thm]{Conjecture}

\theoremstyle{definition}
\newtheorem{defn}[thm]{Definition}

\newtheorem{innercustomthm}{Theorem}
\newenvironment{customthm}[1]
  {\renewcommand\theinnercustomthm{#1}\innercustomthm}
  {\endinnercustomthm}

\theoremstyle{remark}
\newtheorem{rem}[thm]{Remark}

\makeatletter
\let\c@equation\c@thm
\makeatother
\numberwithin{equation}{section}

\makeatletter
\def\subsection{\@startsection{subsection}{3}%
  \z@{.5\linespacing\@plus.7\linespacing}{.1\linespacing}%
  {\bfseries}}
\makeatother

\newcommand{\N} { \mathbb{N}}

\newcommand{\A} { \mathbb{A}}

\newcommand{\R} { \mathbb{R}}

\newcommand{\cut} { \backslash}

\DeclareMathOperator{\AP}{AP}
\DeclareMathOperator{\ap}{ap}

\DeclareMathOperator{\E}{\mathbb{E}}
\renewcommand{\P}{\mathbb{P}}

\usepackage{etoolbox}
\makeatletter
\patchcmd{\@maketitle}
  {\ifx\@empty\@dedicatory}
  {\ifx\@empty\@date \else {\vskip3ex \centering\footnotesize\@date\par\vskip1ex}\fi
   \ifx\@empty\@dedicatory}
  {}{}
\patchcmd{\@adminfootnotes}
  {\ifx\@empty\@date\else \@footnotetext{\@setdate}\fi}
  {}{}{}
\makeatother

\title{$(k,\lambda)$-anti-powers and other patterns in words}
\author{Amanda Burcroff}

\date{ \vspace{-0.3in}{\it University of Michigan\\ 500 S. State St.\\  Ann Arbor, MI 48109 United States}}

\begin{document}
\maketitle
\begin{abstract}
    \vspace{-0.5in}
    Given a word, we are interested in the structure of its contiguous subwords split into $k$ blocks of equal length, especially in the homogeneous and anti-homogeneous cases.  We introduce the notion of $(\mu_1,\dots,\mu_k)$-block-patterns, words of the form $w = w_1\cdots w_k$ where, when $\{w_1,\dots,w_k\}$ is partitioned via equality, there are $\mu_s$ sets of size $s$ for each $s \in \{1,\dots,k\}$.  This is a generalization of the well-studied $k$-powers and the $k$-anti-powers  recently introduced by Fici, Restivo, Silva, and Zamboni, as well as a refinement of the  $(k,\lambda)$-anti-powers introduced by Defant. We generalize the anti-Ramsey-type results of Fici et al. to $(\mu_1,\dots,\mu_k)$-block-patterns and improve their bounds on $N_\alpha(k,k)$, the minimum length such that every word of length $N_\alpha(k,k)$ on an alphabet of size $\alpha$ contains a $k$-power or $k$-anti-power.  We also generalize their results on infinite words avoiding $k$-anti-powers to the case of $(k,\lambda)$-anti-powers.  We provide a few results on the relation between $\alpha$ and  $N_\alpha(k,k)$ and find the expected number of $(\mu_1,\dots,\mu_k)$-block-patterns in a word of length $n$.
\end{abstract}

\blfootnote{{\it E-mail address:} burcroff@umich.edu}
\vspace{-0.5in}
\section{Introduction}
In 1975 Erd{\H o}s, Simonivits, and S{\' o}s \cite{ESS} introduced anti-Ramsey theory, the idea that sufficiently large partitioned structures cannot avoid anti-homogeneous substructures. Their investigation was initially graph-theoretic, but with time anti-Ramsey-type results have permeated many areas of combinatorics, including the studies of Sidon sets, canonical Ramsey theory, and the spectra of colorings \cite{Laci, Rad, SS}. The study of homogeneous and anti-homogeneous substructures can also be extended to {\it words}, finite or infinite (to the right) sequences of letters from a fixed alphabet. The substructures of interest are contiguous subwords, known as {\it factors}.  A well-studied type of regularity in words concerns {\it k-powers}, that is, words of the form $u^k = uu \cdots u$ (concatenated $k$ times) for some nonempty word $u$ (see, for example, \cite{Lot}).  Recently Fici et al. \cite{FRSZ} introduced a notion of anti-regularity in words through their definition of {\it $k$-anti-powers}.  

\begin{defn}
Let $|u|$ denote the length of a word $u$.  A {\it $k$-anti-power} is a word $w$ of the form 
$$w = w_1w_2\cdots w_k$$
such that $|w_1| = \cdots = |w_k|$ and $w_1,\dots,w_k$ are distinct.
\end{defn}

Fici et al. \cite {FRSZ} were able to show several properties of anti-powers in words, including anti-Ramsey results concerning the existence of $\ell$-powers or $k$-anti-powers.  Defant \cite{Def} and Narayanan \cite{Nar} showed that $\ap(t,k)$, the minimum $m > 0$ for which the factor of length $km$ beginning at the first index of the famous Thue-Morse word $t$ is a $k$-anti-power, grows linearly in $k$.  Defant also introduced the notion of $(k,\lambda)$-anti-powers, which is a generalization of $k$-anti-powers.

\begin{defn}
A {\it $(k,\lambda)$-anti-power} is a word $w$ of the form 
$$w = w_1w_2\cdots w_k$$
such that $|w_1| = \cdots = |w_k|$ and $|\{i: w_i = w_j\}| \leq \lambda$ for each fixed $j \in \{1,\dots,k\}$.
\end{defn}

Note that when $\lambda = 1$, this is precisely the definition of a $k$-anti-power.  Whenever such a generalization is nontrivial, we prove that the results of Fici et al. in \cite{FRSZ} concerning $k$-anti-powers generalize to the case of $(k,\lambda)$-anti-powers.  In fact, many of these results can be strengthened by enforcing a particular structure on the partition of the blocks by equality.  We generalize the notions of $k$-powers and $k$-anti-powers while refining the $(k,\lambda)$-anti-powers through the introduction of {\it $(\mu_1,\dots,\mu_k)$-block-patterns}.  
\begin{defn}\label{block patterns}
Let $\mu_1, \dots, \mu_k$ be nonnegative integers satisfying $\sum_{s = 1}^k s\mu_s = k$.  A {\it $(\mu_1, \cdots, \mu_k)$-block-pattern} is a word of the form
$w = w_1\cdots w_k$ where, if the set $\{1,\dots,k\}$ is partitioned via the rule $i \sim j \iff w_i = w_j$, there are $\mu_s$ parts of size $s$ for all $1 \leq s \leq k$.   
\end{defn}

For example, $10\;01\;00\;01\;10$ is a $(1,2,0,0,0)$-block-pattern.  Let $P_{k,\leq \lambda}$ denote the set of $k$-tuples of natural numbers $(\mu_1,\dots,\mu_k)$ such that $\sum_{s = 1}^k s\mu_s = k$ and $\mu_s = 0$ for $s > \lambda$.  These correspond to the partitions of $k$ such that each part has size at most $\lambda$.  We can relate $(\mu_1,\dots,\mu_k)$-block-patterns to $(k,\lambda)$-anti-powers via the following observation.

\begin{rem}
Let $\mathcal{AP}_\A(k,\lambda)$ be the set of $(k,\lambda)$-anti-powers on an alphabet $\A$.  Let $\mathcal{BP}_\A(\mu_1,\dots,\mu_k)$ be the set of $(\mu_1,\dots,\mu_k)$-block-patterns on $\A$.  Then
$$\mathcal{AP}_\A(k,\lambda) = \bigcup_{(\mu_1,\dots,\mu_k) \in P_{k,\leq \lambda}} \mathcal{BP}_\A(\mu_1,\dots,\mu_k).$$
In particular, the $k$-anti-powers are precisely the $(k,0,\dots,0)$-block-patterns, and moreover the $k$-powers are precisely the $(0,\dots,0,1)$-block-patterns.
\end{rem}

The generalizations of the anti-Ramsey results of Fici et al. in \cite{FRSZ} to the case of $(\mu_1,\dots,\mu_k)$-block-patterns are the focus of Section \ref{block general}.  In particular, we obtain bounds on the sizes of words avoiding powers or block-patterns with at most $\sigma$ pairs of equal blocks. In Section \ref{generalize}, we generalize the results of \cite{FRSZ} on avoiding $k$-anti-powers in infinite words to $(k,\lambda)$-anti-powers.  We also observe that Sturmian words have anti-powers of every order starting at each index.

A slight strengthening of the arguments of Fici et al. in \cite{FRSZ} also provide better bounds for $N_\alpha(k,k)$, the smallest positive integer such that every word of length $N_\alpha(k,k)$ over an alphabet of size $\alpha$ contains a $k$-power or $k$-anti-power.  Namely, it is shown in \cite{FRSZ} that for $k > 2$, $k^2 - 1 \leq N_\alpha(k,k) \leq k^3{k \choose 2}$.  In Section \ref{ramsey}, we improve both the lower and upper bounds according to the following theorem.
\begin{customthm}{5.1}
For any $k > 3$, 
$$ 2k^2 - 2k \leq N_\alpha(k,k) \leq (k^3 - k^2 + k){k \choose 2}.$$
\end{customthm}
In Section \ref{ramsey} we also investigate how the size of the alphabet affects $N_\alpha(k,k)$.  In Section \ref{expect}, we return to the more general setting and compute the expected number of $(\mu_1,\dots,\mu_k)$-block-patterns in a word of length $n$.

\section{Preliminaries}\label{prelims}
Let $\N = \{1,2,3,\dots\}$.  The $i^{th}$ letter of a word $x$ is denoted $x[i]$, and for $i < j$ the contiguous substring beginning at the $i^{th}$ letter and ending with the $j^{th}$ is denoted $x[i..j]$.  A word $v$ is a {\it factor} of $x$ if $x = uvw$ for words $u$ and $w$.  In the case that $u$ is empty, $v$ is a {\it prefix} of $x$, and if $w$ is empty, then $v$ is a {\it suffix} of $x$. The suffix of $x$ beginning at the $j^{th}$ index of $x$ is denoted $x_{(j)}$.  If $w$ is both a prefix and suffix of $x$, then $w$ is a {\it border} of $x$.  

A word is called {\it recurrent} if every finite factor appears infinitely many times in the word.  A word $x$ is called {\it eventually periodic} if there exists an index $j \geq 0$ and a finite word $u$ such that $x_{(j)} = u^\omega$; otherwise $x$ is called {\it aperiodic}.  A word is called {\it $\omega$-power-free} if for every finite factor $u$, there exists an $\ell \in \N$ such that $u^\ell$ is not a factor.  Note that a word that avoids $k$-powers for some $k \in \N$ is $\omega$-power-free, but the converse is not necessarily true.

Let $[\alpha] = \{1,\dots,\alpha\}$.  The {\it lower density } and {\it upper density }  of a subset $S$ of $\N$ are given respectively by 
$$\underline{d}(S) = \liminf_{n \to \infty} \frac{|S \cap [n]|}{n} \;\;\;\text{ and }\;\;\; \overline{d}(S) = \limsup_{n \to \infty} \frac{|S \cap [n]|}{n}.$$

\section{Generalization of an Anti-Ramsey result to $(\mu_1,\dots,\mu_k)$-block-patterns}\label{block general}

A main result of Fici et al. \cite{FRSZ} is that every infinite word contains either powers of all orders or anti-powers of all orders.  Since powers are homogeneous substructures whereas the anti-powers are anti-homogeneous, one may wonder if similar results can be demonstrated for substructures between these extremes.  We will generalize their result to the case of $(\mu_1,\dots,\mu_k)$-block-patterns in infinite words.  The density bounds rely on the number of pairs of equal blocks that are forced in the prefixes of length $km, \dots, k(m + \beta)$ for some $m,\beta$.  The following definition is created to account for these pairs.

\begin{defn}
Let $D(x,k,\sigma)$ be the set of $m \in \N$ such that the prefix of the word $x$ of length $km$ is a $(\mu_1,\dots,\mu_k)$-block-pattern satisfying $\sum_{s = 1}^k \mu_s{s \choose 2} \leq \sigma$.
\end{defn}

Note that $D(x,k,\sigma)$ is closed downward with respect to the dominance order.  That is, if $m \in D(x,k,\sigma)$ and $m' \in \N$ are such that $x[1..km]$ is a $(\mu_1,\dots,\mu_k)$-block-pattern and $x[1..km']$ is a $(\mu'_1,\dots,\mu'_k)$-block-pattern satisfying $\sum_{s = 1}^\ell \mu_s \geq \sum_{s = 1}^\ell \mu'_s$ for each $\ell \in \{1,\dots,k\}$, then $m' \in D(x,k,\sigma)$.

For the proof of Theorem \ref{inf powers or patterns}, we make use of the following lemma of Fici, Restivo, Silva, and Zamboni.  
\begin{lem}\label{border implies power}
\emph{(\cite{FRSZ}, Lemma 3)}
Let $v$ be a border of a word $w$ and let $u$ be the word such that $w = uv$.  If $\ell$ is an integer such that $|w| \geq \ell |u|$, then $u^\ell$ is a prefix of $w$.  
\end{lem}

\begin{thm}\label{inf powers or patterns}
Let $x$ be an infinite word such that 
$$\overline{d}(D(x,k,\sigma)) \geq \left( 1 + \left\lfloor \frac{1}{\sigma}{k \choose 2}\right\rfloor\right)^{-1}$$
for some $k, \sigma \in \N$.  For every $\ell$, there is a word $u$ with $|u| \leq (k-1)\left\lfloor\frac{1}{\sigma}{k \choose 2}\right\rfloor$ such that $u^\ell$ is a factor of $x$.
\end{thm}
\begin{proof}
Fix such a $k$ and $\sigma$.  Fix an arbitrary $\ell \geq 1$, and let $\beta = \left\lfloor\frac{1}{\sigma}{k \choose 2}\right\rfloor$.  By the condition on the upper density of $D(x,k,\sigma)$, there exists some integer $m > \ell(k-1)\beta$ such that $\{m,m+1,\dots,m + \beta\} \subset D(x,k,\sigma)$.  Following \cite{FRSZ}, for every $j \in \{0,\dots,k-1\}$ and $r \in \{m,\dots, m + \beta\}$, set
$$U_{j,r} = x[jr + 1..(j+1)r].$$
That is, $U_{0,r}\cdots U_{k-1,r} = x[1..kr]$.  Since $\{m,m+1,\dots,m + \beta\} \subset D(x,k,\sigma)$, we are guaranteed at least $(\beta+1)\sigma > {k \choose 2}$ triples $(i,j,r)$ such that $i < j$ and $U_{i,r} = U_{j,r}$.  By the Pigeonhole Principle, there exist $i,j,r,s$ such that $0 \leq i < j \leq k-1$, $m \leq r < s \leq \beta + 1$, $U_{i,r} = U_{j,r}$, and $U_{i,s} = U_{j,s}$.  

Setting $w = x[is + 1..(i+1)r]$ and $v = x[js + 1..(j+1)r]$, we have
$$|v| = (j + 1)r - js < (i + 1)r - is = |w|,$$
so $v$ is a border of $w$.  Writing $w = uv$, we have
$$1 \leq |u| = |w| - |v| = (j-i)(s-r) \leq (k-1)\beta$$
while 
$$|v| = r - j(s-r) \geq m - (k-1)\beta \geq \ell (k-1)\beta \geq (\ell - 1)|u|.$$
Hence, $|w| = |u| + |v| \geq \ell|u|$.  By Lemma \ref{border implies power}, $u^\ell$ is a factor of $x$.
\end{proof}

Theorem \ref{inf powers or patterns} can be applied to the special case of $(k,\lambda)$-anti-powers.  The definition of $(k,\lambda)$-anti-powers suggests the following generalization of $\AP(x,k)$, the set of integers $m$ such that the prefix of $x$ of length $km$ is a $k$-anti-power.

\begin{defn}
Let $\AP(x,k,\lambda)$ be the set of $m \in \N$ such that the prefix of the word $x$ of length $km$ is a $(k,\lambda)$-anti-power.
\end{defn}
Note that $\AP(x,k,1) = \AP(x,k)$.

\begin{cor}\label{inf powers or k lambda}
Let $x$ be an infinite word such that 
$$\underline{d}(\AP(x,k,\lambda)) < \left(1 + \left\lfloor \frac{k^2 - k}{\lambda^2 + \lambda} \right\rfloor \right)^{-1}$$
for some $k, \lambda \in \N$.  For every $\ell$, there is a word $u$ with $|u| \leq (k-1)\left\lfloor \frac{k^2 - k}{\lambda^2 + \lambda} \right\rfloor$ such that $u^\ell$ is a factor of $x$.
\end{cor}
\begin{proof}
Fix $k$ and $\lambda$ as above.  Note that $\N \cut \AP(x,k,\lambda) \subseteq D\left(x,k,{\lambda + 1 \choose 2}\right)$.  Hence, 
$$\underline{d}(\AP(x,k,\lambda)) < \left(1 + \left\lfloor \frac{k^2 - k}{\lambda^2 + \lambda} \right\rfloor \right)^{-1}$$
implies
$$\overline{d}\left(D\left(x,k,{\lambda + 1 \choose 2}\right)\right) \geq \frac{\left\lfloor \frac{k^2 - k}{\lambda^2 + \lambda} \right\rfloor}{1 + \left\lfloor \frac{k^2 - k}{\lambda^2 + \lambda} \right\rfloor} \geq \frac{1}{1 + \left\lfloor{\lambda + 1 \choose 2}^{-1} {k\choose 2}\right\rfloor}.$$
This shows that $x$ satisfies the conditions of Theorem \ref{inf powers or patterns} for the same $k$ and $\sigma = {\lambda + 1 \choose 2}$.
\end{proof}

In the case that our alphabet is finite, there are finitely many factors of length at most $(k-1)\left\lfloor \frac{k^2 - k}{\lambda^2 + \lambda} \right\rfloor$.  Thus, the Pigeonhole Principle allows us to choose a word $u$ that works for every $\ell$ in Theorem \ref{inf powers or k lambda}.

\begin{cor}\label{inf powers fin alphabet}
Let $x$ be an infinite word on a finite alphabet such that 
$$\underline{d}(\AP(x,k,\lambda)) < \left(1 + \left\lfloor \frac{k^2 - k}{\lambda^2 + \lambda} \right\rfloor \right)^{-1}$$
for some $k, \lambda \in \N$.  There is a word $u$ with $|u| \leq (k-1)\left\lfloor \frac{k^2 - k}{\lambda^2 + \lambda} \right\rfloor$ such that $u^\ell$ is a factor of $x$ for every $\ell > 0$. In particular, $x$ is not $\omega$-power-free.
\end{cor}

There is a $\lambda = 1$ analogue to Corollary \ref{inf powers or k lambda} in \cite{FRSZ} (their Theorem 4), which claims under the same density condition that $x$ is not $\omega$-power-free.  Though the condition that the alphabet is finite is not explicitly stated, their result is false for infinite alphabets.  In fact, there exist $\omega$-power-free words which avoid $k$-anti-power prefixes for some fixed $k \in \N$. These words also show that Theorem 6 of \cite{FRSZ}, which states that $\omega$-power-free words have anti-powers of every order beginning at each index, is false when infinite alphabets are allowed.  Theorem \ref{counterexample} provides a counterexample to Theorems 4 and 6 in \cite{FRSZ} when infinite alphabets are permitted.

\begin{thm}\label{counterexample}
There exists an $\omega$-power-free word $x$ on an infinite alphabet such that $\AP(x,k)$ is empty for some $k \in \N$.
\end{thm}
\begin{proof}
Let $y = \prod_{i = 1}^\infty (a_i)^{2^i}$.  Since there are finitely many appearances of each letter $a_i$, $y$ is clearly $\omega$-power-free.  Note that if $2^{i + 1} - 2^i = 2^i \geq 4m$ for some block length $m$ and some $i$ satisfying $2^{i + 1} < km$, then two blocks of the prefix of length $km$ must equal $a_i^m$.  Hence, $m \not\in \AP(x,k)$.  For $k \geq 17$, such an $i$ always exists.   $\AP(x,k)$ is empty for $k \geq 17$, despite $x$ being $\omega$-power-free.  
\end{proof}

We return to a modified version of the proof of Theorem \ref{inf powers or patterns} in order to find bounds on the length of words avoiding $k$-powers and $k$-anti-powers.

\begin{thm}\label{finite block bounds}
For all integers $\ell > 1, k > 1, \sigma \geq 1$ there exists $N'_\alpha(\ell,k,\sigma)$ such that every word of length $N'_\alpha(\ell,k,\sigma)$ on $[\alpha]$ contains an $\ell$-power or $(\mu_1,\dots,\mu_k)$-block-pattern satisfying $\sum_{s = 1}^k \mu_s{s \choose 2} \leq \sigma$.  Moreover,
$$k\left(k - \left\lfloor \frac{1}{2}(\sqrt{8\sigma + 1}  + 1)\right\rfloor\right) \leq N'_\alpha(k,k,\lambda) \leq \left\lfloor\frac{1}{\sigma}{k \choose 2}\right\rfloor(k^3 - k^2 + k).$$
\end{thm}
\begin{proof}
As in \cite{FRSZ}, the upper bound follows from the proof of the infinite case in Theorem \ref{inf powers or patterns}.  Let $\beta = \left\lfloor\frac{1}{\sigma}{k \choose 2}\right\rfloor$.  Let $x$ be any word of length $\beta(k^3 - k^2 + k)$.  For each $r \in \{(k^2 -k)\beta,\dots,(k^2 - k + 1)\beta\}$, consider the first $k$ consecutive blocks of length $r$ in $x$, denoted by $U_{0,r},U_{1,r},\dots, U_{k-1,r}$.  If $x$ does not contain any element of $D(x,k,\sigma)$, then there exist $i,j,r,s$ such that $0 \leq i < j \leq k-1$, $m \leq r < s \leq \beta + 1$, $U_{i,r} = U_{j,r}$ and $U_{i,s} = U_{j,s}$.  Setting $w = x[is + 1..(i+1)r]$ and $v = x[js + 1..(j+1)r]$, we have that $v$ is a border of $w$.  Writing $w = uv$, we have $|u| \leq (k-1)\beta$ and 
$$|w| = |u| + |v| \geq |u| + r - j(s-r) \geq |u| + (k-1)^2\beta \geq k|u|.$$
By Lemma \ref{border implies power}, we get that $u^k$ is a factor of $x$, i.e., $x$ contains a $k$-power.  The length of $x$ is chosen to accommodate $k$ blocks of size at most $(k^2 - k + 1)\beta$.

The lower bound is proven via a construction; we will show that the word 
$$x = 0^{k-1}(10^{k-1})^{ k - \left\lfloor \frac{1}{2}(\sqrt{8\sigma + 1}  + 1)\right\rfloor - 1}$$
avoids $k$-powers and $(\mu_1,\dots,\mu_k)$-block-patterns with $\sum_{s = 1}^k \mu_s{s \choose 2} \leq \sigma$.   Since $\sigma \geq 1$, we have
$$k - \left\lfloor \frac{1}{2}(\sqrt{8\sigma + 1} + 1)\right\rfloor - 1  \leq k - 1.$$
If $u^k$ were a factor of $x$, either $u$ would contain the letter $1$, contradicting the fact that that $x$ has at most $k-1$ copies of the letter $1$, or $u = 0^m$ for some $m \geq 1$, contradicting the fact that $x$ has no factor equal to $0^k$.  Hence, $x$ avoids $k$-powers.  We can see that for every factor $v$ of length $km$, at least $\left\lfloor \frac{1}{2}(\sqrt{8\sigma + 1} + 1)\right\rfloor + 1$ blocks of $v$ are equal to $0^m$.   $v$ is a $(\mu_1,\dots,\mu_k)$-block-pattern with 
$$\sum_{s = 1}^k \mu_s{s \choose 2} \geq {\left\lfloor \frac{1}{2}(\sqrt{8\sigma + 1} + 1)\right\rfloor + 1  \choose 2} > \frac{1}{8}(\sqrt{8\sigma + 1} + 1)(\sqrt{8\sigma + 1}) \geq \sigma.$$
\end{proof}

We can specialize Theorem \ref{finite block bounds} to the case of $(k,\lambda)$-anti-powers.

\begin{cor}\label{lambda bounds}
For all integers $\ell > 1, k > 1, \lambda \geq 1$, there exists $N_\alpha(\ell,k,\lambda)$ such that every word of length $N_\alpha(\ell,k,\lambda)$ on $[\alpha]$ contains an $\ell$-power or $(k,\lambda)$-anti-power.  Moreover,
$$k(k-\lambda) \leq N_\alpha(k,k,\lambda) \leq \left\lfloor \frac{k^2 - k}{\lambda^2 + \lambda} \right\rfloor(k^3 - k^2 + k).$$
\end{cor}
\begin{proof}
Suppose a word avoids $(k,\lambda)$-anti-powers.  Then it avoids $(\mu_1,\dots,\mu_k)$-block-patterns with $\sum_{s = 1}^k \mu_s{s \choose 2} \leq {\lambda + 1 \choose 2}$.  Applying Theorem \ref{finite block bounds} with $\sigma = {\lambda + 1 \choose 2}$ yields the corresponding bounds.
\end{proof}

In particular, this improves upon the upper bound for $N_\alpha(k,k)$ (in their notation, $N(k,k)$) in \cite{FRSZ}.

\begin{cor}\label{improved upper bound}
For all $k > 1$,
$$N_\alpha(k,k) \leq (k^3 - k^2 + k){k \choose 2}.$$
\end{cor}

\section{Avoiding Anti-Powers}\label{generalize}
  This section is devoted to generalizing the results of Fici et al. \cite{FRSZ} on infinite words avoiding $k$-anti-powers to the case of $(k,\lambda)$-anti-powers. Many of these generalizations can be achieved using proofs similar to those in \cite{FRSZ}. We also provide a condensed proof of the fact that the Sturmian words contain anti-powers of every order beginning at every index.
 
  We begin with a straightforward lemma.
\begin{lem}\label{step down k lambda}
Suppose $k > \lambda > j > 1$. If a word avoids $(k,\lambda)$-anti-powers, then it avoids $(k-j,\lambda - j)$-anti-powers. 
\end{lem}
\begin{proof}
It is enough to show that if a word avoids $(k,\lambda)$-anti-powers, then it avoids $(k-1,\lambda - 1)$-anti-powers.  Suppose that a word $x$ contains a $(k-1,\lambda -1)$-anti-power $w$ of length $km$.  If we extend to the right by $m$ letters, we obtain a  $(k,\lambda)$-anti-power, since we increase the number of equal blocks, $|\{i : w_i = w_j\}|$ for any $j$, by at most $1$. 
\end{proof}
\begin{defn}
We call an infinite word {\it constant} if it is of the form $a^\omega$ for some $a \in \A$.
\end{defn}
In order to classify the words avoiding $(k,k-2)$-anti-powers, we will use two results of Fici, Restivo, Silva, and Zamboni.
\begin{lem}\label{three binary}\emph{(\cite{FRSZ}, Lemma 9)}
Let $x$ be an infinite word.  If $x$ avoids $3$-anti-powers, then $x$ is a binary word.
\end{lem}
\begin{prop}\label{three factor}\emph{(\cite{FRSZ}, Proposition 10)}
Let $x$ be an infinite word.  If $x$ avoids $3$-anti-powers, then it cannot contain a factor of the form $10^n1$ or $01^n0$ with $n > 1$.
\end{prop}

\begin{thm}\text{ }\\\vspace{-0.5cm}
\begin{enumerate}
    \item For $k > 1$, the infinite words avoiding $(k,k-1)$-anti-powers are precisely the constant words.
    \item For $k > 2$, infinite words avoiding $(k,k-2)$-anti-powers are the words that differ from a constant word in at most one position.
    \item For $k > 3$, there exist infinite aperiodic words avoiding $(k,k-3)$-anti-powers.
\end{enumerate}
\end{thm}
\begin{proof}
The first claim is trivial; merely note that the avoidance of $(k,k-1)$-anti-powers implies that every factor whose length is a multiple of $k$ is a $k$-power.

For the second claim, let $x$ be a word avoiding $(k,k-2)$-anti-powers.  By Lemma \ref{step down k lambda}, $x$ avoids $3$-anti-powers.  By Lemma \ref{three binary}, $x$ is a binary word.  Suppose, seeking a contradiction, that $x$ has at least $2$ instances of $1$ and at least $2$ instances of $0$, i.e., $x$ differs from a constant word in more than one position.  Then $x$ has a factor of the form $10^a1^b0$ or $01^a0^b1$ for some $a \geq 1, b \geq 1$; without loss of generality assume it is the first.  By Proposition \ref{three factor}, $a = b = 1$.  However, under these conditions, $x$ has a factor of the form $1010$, which is itself a $(4,2)$-anti-power.

For the third claim, we exhibit a family of infinite aperiodic words avoiding $(k,k-3)$-anti-powers.  Let $\{\gamma_i\}_{i = 1}^n$ be an increasing sequence such that $\gamma_{i + 1} \geq (k+1)\gamma_i$ for all $i \in \N$.  Define a word $x$ as follows:
$$x[j] = \begin{cases} 1 & \text{if $j = \gamma_i$ for some $i$;}\\
0 & \text{otherwise.}\end{cases}$$
We will show that $x$ avoids $(k,k-3)$-anti-powers.  Note that if $x[\ell + 1..\ell + n]$ has at least two nonzero entries, then for some $i$ we have
$$\ell + 1 \leq \gamma_i < (k + 1) \gamma_i \leq \gamma_{i + 1} \leq \ell + n.$$
This implies that $n > k\gamma_i \geq k(\ell + 1)$, so $\ell + 1 \leq \frac{n}{k}$.  Suppose, seeking a contradiction, that the $k$ consecutive blocks $x[j + 1..j + m], \dots, x[j + (k-1)m + 1.. j + km]$ form a $(k,k-3)$-anti-power.  At most $k-3$ of these blocks can be $0^m$, so the word $x[j + m + 1.. j + km]$ has at least two nonzero entries.  Thus, $j + m + 1 \leq \frac{(k-1)m}{k}$.  It follows that $j + 1 < 0$, a contradiction.
\end{proof}

\begin{thm}
For all $k \geq 6$, there exist aperiodic recurrent words avoiding $(k, k-5)$-anti-powers.
\end{thm}
\begin{proof}
Let $w$ be the limit of the sequence $w_0 = 0$, $w_{n + 1} = w_n1^{(k-3)|w_n|}w_n$.  Note that each occurrence of $w_n$ except the first is preceded and followed by $1^{(k-3)|w_n|}$.  Let $v = v_1v_2\cdots v_k$ be a factor of $w$, where $|v_i| = \ell > 0$ for all $i \in \{1,\dots,k\}$.  Let $n$ be the largest integer such that 
$$|w_n| = (k-1)^n < 2\ell < (k-1)^{n+1} = |w_{n+1}|.$$
Since $w$ is recurrent, we can assume $v$ appears after the first appearance of $w_n$.  

We claim that at most four blocks of $v$ can intersect an occurrence of $w_n$.  Each occurrence of $w_n$ intersects at most two blocks of $v$ by the condition $2\ell > |w_n|$.  Moreover, any three occurrences of $w_n$ are separated by factors of $1^{(k-3)|w_n|}$ and $1^{(k-3)|w_{n+1}|},$.  As
$$|v| = k\ell < \frac{k}{2}|w_{n + 1}| \leq (k-3)|w_{n+1}|,$$
$v$ can intersect at most $2$ occurrences of $w_n$.  We can conclude that at most four blocks of $v$ are not equal to $1^\ell$.
\end{proof}

We now restrict ourselves to the setting of $k$-anti-powers.  In \cite{FRSZ}, Fici et al. question under what conditions aperiodic recurrent words can avoid $k$-anti-powers.  It is known this is possible for $k \geq 6$ and impossible for $k \leq 3$, but nothing has been shown for $k = 4$ or $5$.  One class of aperiodic recurrent words that we can exclude from this search are the {\it Sturmian words}.  
\begin{defn}
A {\it Sturmian word} is an infinite word $x$ such that for all $n \in \N$, $x$ has exactly $n +1$ distinct factors of length $n$.
\end{defn}
Note that Sturmian words are necessarily binary.  An alternate characterization of the Sturmian words in terms of {\it irrationally mechanical words} was given by Morse and Hedlund \cite{MH} in 1938.

\begin{defn}
The {\it upper mechanical word} $s_{\theta,x}$ and the {\it lower mechanical word} $s'_{\theta,x}$ with angle $\theta$ and initial position $x$ are defined, respectively, by
\begin{align*}
    s_{\theta,x}[n] &= \begin{cases} 1 & \text{ if } \theta (n-1) + x \in [1-\theta, 1) \mod 1\\ 0 & \text{ if } \theta (n-1) + x\in [0,1-\theta) \mod 1\end{cases}\\
     s'_{\theta,x}[n] &= \begin{cases} 1 & \text{ if } \theta (n-1) + x \in (1-\theta, 1) \mod 1\\ 0 & \text{ if } \theta (n-1) + x\in [0,1-\theta] \mod 1\end{cases}
\end{align*}
for some $\theta, x\in \R$.  A word $w$ is called {\it irrationally mechanical} if $w = s_{\theta,x}$ or $w = s'_{\theta,x}$ for some $x \in \R$ and irrational $\theta \in \R$.
\end{defn}

\begin{thm}\emph{(\cite{MH})}
A word is Sturmian if and only if it is irrationally mechanical.
\end{thm}

Irrationally mechanical words can be interpreted through the lens of mathematical billiards.  Consider the unit circle centered at the origin, parameterized by $g(t) = (\cos(2\pi t),\sin(2\pi t))$ for $t \in \R$.  Place an (infinitesimal) ball at point $g(x)$ on the circle and shoot it in a straight trajectory toward $g(x + \pi)$.  At each moment the ball "bounces off" the circle, it generates a $0$ if it hits the point $g(x)$ for $x \in [0,1-\theta)$ and a $1$ otherwise.  The sequence generated by the trajectory of such a ball is precisely the word $s_{\theta,x}$.  For example, a trajectory generating the famous Fibonacci word is shown below.

\begin{figure}[h]
    \centering
    \includegraphics[scale=0.1]{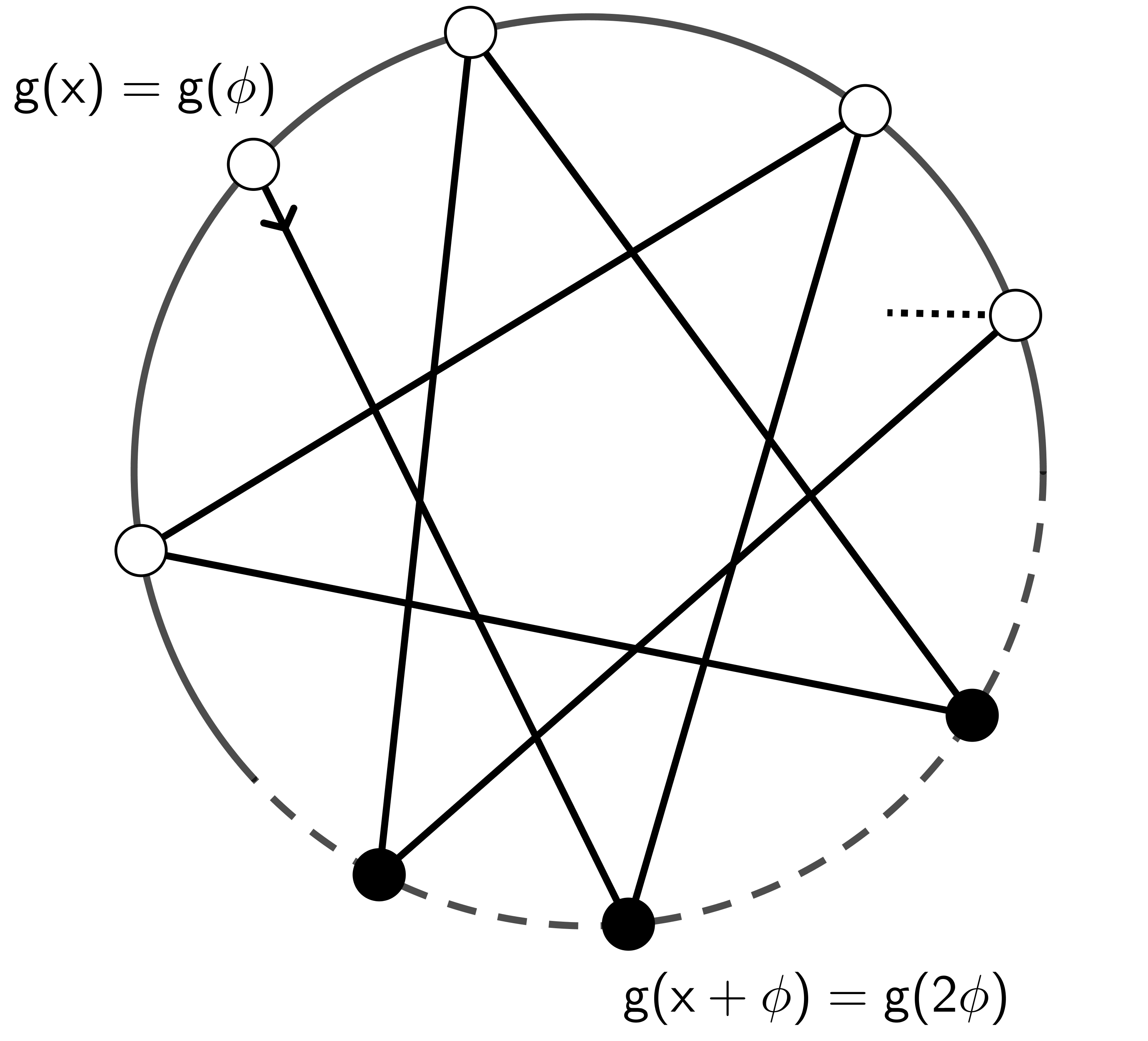}
    \caption{The trajectory associated with the Fibonacci word $s_{\phi,\phi} = 01001010...$, where $\phi$ is the golden ratio $1.6180339...$.  A white point indicates that the letter $0$ is generated, and a black point indicates that the letter $1$ is generated.  The Fibonacci word can also be generated as the limit of the sequence $\{S_n\}_{n = 1}^\infty$, where $S_1 = 0$, $S_2 = 01$, and $S_n = S_{n - 1}S_{n-2}$ for $n \geq 3$.}
    \label{fib}
\end{figure}

The Sturmian words comprise a well-studied class of aperiodic recurrent words.  We will show that for any Sturmian word $x$, $j \geq 0$, and $k \geq 1$, $x$ contains a $k$-anti-power beginning at $x[j]$.  Hence, the Sturmian words cannot avoid $k$-anti-powers for any $k \geq 1$.  It is enough to show that the Sturmian words are $\omega$-power-free by the following theorem from \cite{FRSZ}.

\begin{thm}\emph{(\cite{FRSZ}, Theorem 6)}
Let $x$ be an $\omega$-power-free word on a finite alphabet.  For every $k > 1$ there is an occurrence of a $k$-anti-power starting at every position of $x$.
\end{thm}

In fact, for every Sturmian word $x$ there exists an $M \in \N$ such that $x$ avoids $M$ powers.  This follows from the work of Fici, Langiu, Lecroq, Lefebvre, Mignosi, Peltom{\" a}ki, and Prieur-Gaston in \cite{FLLLMPP}.  They prove a stronger but somewhat lengthy result about a generalized notion of powers known as {\it abelian powers}.  We provide a condensed proof of our weaker claim.

\begin{thm}
Let $w$ be a Sturmian word with angle $\theta$.  Let $M =\left \lceil (\min\{\theta, 1-\theta\})^{-1}\right \rceil + 1$.  Then $s_\theta$ avoids $M$-powers.
\end{thm}
\begin{proof}
Suppose $w$ is an upper mechanical word; the case for lower mechanical words follows analogously.  Let $w = s_{\theta,x}$ for $x \in \R$.  Suppose, seeking a contradiction, that there exists a factor $u$ of length $m$ such that $u^{M}$ is also a factor.  Let $\{z\}$ denote the fractional part of $z \in \R$.  There is some nonnegative integer $r$ such that 
$$\{x + r\theta\}, \{x + (r+m)\theta\}, \{x + (r + 2m)\theta\}, \cdots, \{x + (r + (M - 1)m)\theta\}$$
either all lie in the interval $[0, 1 - \theta)$ or all lie in the interval $[1 - \theta, 1)$.

In the first case, we divide $[0,1)$ up into $\lceil \theta^{-1}\rceil$ intervals of uniform size, starting at $0$.  By the Pigeonhole Principal, at least two of the $M$ points lie in the same interval.  In other words, there exist $0 \leq q, i , j \leq \lceil \theta^{-1}\rceil$ such that 
$$\{(x + (r + im)\theta\}, \{x + (r + jm)\theta\} \in \bigg[ \frac{q}{\lceil \theta^{-1}\rceil}, \frac{q + 1}{\lceil \theta^{-1}\rceil} \bigg).$$
Hence, $\{x + (r - |i - j|m)\theta\} \in \big[ 1 - \lceil \theta^{-1}\rceil^{-1}, 1\big) \subseteq [1 - \theta,1)$, a contradiction.

In the second case, we divide $[0,1)$ up into $\lceil (1-\theta)^{-1}\rceil$ intervals of uniform size, starting at $0$.  By the Pigeonhole Principal, at least two of the $M$ points lie in the same interval.  In other words, there exist $0 \leq q, i , j \leq \lceil (1-\theta)^{-1}\rceil$ such that 
$$\{(x + (r + im)\theta\}, \{x + (r + jm)\theta\} \in \bigg[ \frac{q}{\lceil (1-\theta)^{-1}\rceil}, \frac{q + 1}{\lceil (1-\theta)^{-1}\rceil} \bigg).$$
Hence, $\{x + (r + |i - j|m)\theta\} \in \big[ 0, \lceil (1-\theta)^{-1}\rceil^{-1}\big) \subseteq [0,1-\theta)$, a contradiction.
\end{proof}

\begin{cor}
Let $x$ be a Sturmian word.  For every $k > 1$, there is an occurrence of a $k$-anti-power starting at every position of $x$.
\end{cor}

\begin{defn}
Given a sequence $\{v_n\}_{n = 1}^\infty$ of finite words, define words $w_n$ by $w_1 = v_1$ and $w_{n+1} = w_nv_nw_n$.  The limit of the sequence of words $\{w_n\}_{n = 1}^\infty$ is called the {\it sesquipower induced by the sequence $\{v_n\}_{n = 1}^\infty$}.
\end{defn}
It is well-known that an infinite word is recurrent if and only if it is a sesquipower (see, for example, \cite{Lot}).  We will show that if an aperiodic recurrent word avoids $k$-anti-powers, then we can deduce some properties about the sequence $\{v_n\}_{n = 1}^\infty$.
\begin{thm}
Let $x$ be the aperiodic sesquipower on a finite alphabet induced by $\{v_n\}_{n=1}^\infty$, and suppose $x$ avoids $k$-anti-powers for some $k \geq 2$.  There exists a word $u$ of length at most $k-1$ such that for all $\ell > 0$, there is some $n > 0$ such that $u^\ell$ is a factor of $v_n$.
\end{thm}
\begin{proof}
Since $x$ avoids $k$-anti-powers, Corollary \ref{inf powers fin alphabet} implies $x$ is not $\omega$-power-free.  Thus, there is some factor $u$ of $x$ such that $u^\ell$ is a factor for every $\ell > 0$.  We can assume $u$ is not an $m$-power for any $m \geq 2$; otherwise, let $u = u_{(|u|/m)}$.  Suppose, seeking a contradiction, that $|u| \geq k$.  Note that the prefix of length $k(|u| + 1)$ of $u^{k + 1}$ is a $k$-anti-power, contradicting the fact that $u^{k+1}$ is a factor of $x$ that avoids $k$-anti-powers.  Thus, $|u| \leq k-1$.

We now know arbitrarily long powers of $u$ occur in $x$, but, in fact, we can show that arbitrarily long powers of $u$ occur in $\{v_n\}_{n = 1}^\infty$.  Since $x$ is not periodic, there exists a $k$ such that $v_k$ is not a factor of $u^\omega$.  Let $\ell_0$ be the largest power of $u$ that is a factor of $x_k$.  For sufficiently large $\ell$, there is some $m \geq k$ such that  $u^\ell$ is a factor of $w_{m+1}$ but not of $w_m$.  We can conclude $u^{\ell - 2\ell_0 - 2}$ is a factor of $v_{m + 1}$, as $w_k$ is a border of $w_m$.  As $\ell_0$ is fixed, this implies that for all $\ell > 0$, $u^\ell$ is a factor of some $v_m$.
\end{proof}

To summarize, if $x$ is an aperiodic recurrent word avoiding $k$-powers for some $k \geq 2$, then $x$ is non-Sturmian and is the sesquipower induced by a sequence $\{v_n\}_{n=1}^\infty$ where the $v_n$ contain arbitrarily long powers of some word $u$.

\section{Avoiding Powers and Anti-Powers}\label{ramsey}

In \cite{FRSZ}, Fici et al. show that for every $\ell,k > 1$, there exists $N_\alpha(\ell,k)$ such that every word of length $N_\alpha(\ell,k)$ on an alphabet of size $\alpha$ contains either an $\ell$-power or a $k$-anti-power.  They prove that for $k > 2$, one has $k^2 - 1 \leq N_\alpha(k,k) \leq k^3{k \choose 2}$.  We improve both these lower and upper bounds.

\begin{thm}\label{nkk bounds}
For any $k > 3$, 
$$ 2k^2 - 2k \leq N_\alpha(k,k) \leq (k^3 - k^2 + k){k \choose 2}.$$
\end{thm}
\begin{proof}
The upper bound is precisely the statement of Corollary \ref{improved upper bound}.

For the lower bound, consider the word 
$$x = 1(0^{k-1}1)^{k-2}0^{k-2}10^{k-2}(10^{k-1})^{k-2}1.$$
We begin by showing that the border $1(0^{k-1}1)^{k-2}0^{k-2}10^{k-2}$ of length $k^2 - 2$ avoids $k$-powers and $k$-anti-powers.  In their proof that $k^2 - 1 \leq N_\alpha(k,k)$, Fici et al. \cite{FRSZ} show that the word $(0^{k-1}1)^{k-2}0^{k-2}10^{k-1}$ avoids $k$-powers and $k$-anti-powers, so we need only check this border for $k$-power or $k$-anti-power prefixes.  We can see immediately that there are no $k$-power prefixes: as $1 \leq m \leq k-1$, the first block of length $m$ of the prefix of length $km$ begins with $1$ while the second begins with $0$.  Suppose, seeking a contradiction, that the prefix of length $km$ is a $k$-anti-power for some $m$.  Since the prefix of length $km$ would need to contain at least $k-1$ instances of the letter $1$ to distinguish the blocks, we require $km \geq 1 + k(k-2)$.  Hence, $m \geq k -1$. We know the block length $m$ is at most $k-1$ since $1(0^{k-1}1)^{k-2}0^{k-2}10^{k-2}$ has length $k^2 - 2$.  However, as $k(k-2) + 1 = (k - 1)^2$, the last two blocks must be $0^{k-2}1$.  The equality of these blocks contradicts the assumption that the prefix of length $km$ is a $k$-anti-power.  $1(0^{k-1}1)^{k-2}0^{k-2}10^{k-2}$ avoids both $k$-powers and $k$-anti-powers.

Thus, we need to consider only those factors of $x$ intersecting nontrivially with the prefix and suffix of length $k^2 - 2$.  Fix such a factor $y$ of length $km$ starting at position $j$.

Suppose, seeking a contradiction, that $y$ is a $k$-power. Let $y_\ell = x[j + \ell m..j + (\ell+1)m - 1]$ be the $\ell^{th}$ block of length $m$ in $y$.  Choose $b$ such that the central letter $1$ of $x$ is contained in $y_b$.  That is, $j + bm \leq k(k-1) \leq j + (b+1)m - 1$. Since $k \geq 4$ and there are exactly two occurrences of the factor $10^{k-2}1$ in $x$, the block $y_b$ cannot contain $10^{k-2}1$ as a factor.  Note $y_\ell \not= 0^m$ for any nonnegative integers $\ell$ and $m$.  As $k \geq 4$, one of $b-2, b+2 \in \{0,\dots,k-1\}$; without loss of generality assume it is $b+2$.  Thus, the two factors $y_{b-1}y_b$ and $y_{b+1}y_{b+2}$ of $x$ each contain an occurrence of the factor $10^{k-2}1$.  However, the only two occurrences of $10^{k-2}1$ in $x$ intersect while $y_{b-1}y_b$ and $y_{b+1}y_{b+2}$ are disjoint, so we've reached a contradiction.  Therefore, $x$ avoid $k$-powers.

Now we show that such a factor $y$ is not a $k$-anti-power.  Suppose it were.  Since $y$ contains at least $k-1$ occurrences of the letter $1$, it follows that $m \geq k - 2$.  In the case $m = k - 2$, each block contains at most one occurrence of the letter $1$, but there are only $k-1$ distinct such blocks.  One can check $m \not= k - 1, k$ by examining the period of the prefix/suffix of length $k(k-2)+1$ or the middle section of length $2k-1$.  Thus, taking into consideration the length of $x$, we have $k + 1 \leq m \leq 2k - 3$.  Consider all blocks except $y_b$ (the block containing the central 1).  Note that the letters of each block are determined by the number of leading $0$'s, which is at most $k - 1$.  If there are $\ell$ blocks preceding $y_b$, and $y_{b-1}$ has $z$ leading zeros, then the numbers of leading zeros for all blocks except $y_b$ are given by the multiset
\[\tag{$*$}\label{blockfronts}\{z + (\ell -1)m, \dots, z + m, z, z - 2m - 2, z - 3m - 2, \dots, z - (k-\ell)m - 2\} \mod k\]
which has a repeated element if and only if the multiset
$$\{(l-1)m,\dots,m,0, (k-2)m-2, (k-3)m-2,\dots, \ell m - 2\} \mod k$$
has a repeated element. Assuming this has no repeated element, we have
$$\{\ell m - 2, \dots, (k-2)m-2\} \subseteq \{\ell m, (\ell + 1)m,\dots, (k-1)m\}.$$
The left-hand side has size $k - \ell - 1$ while the right-hand side has size at most $k - \ell$, and both are arithmetic progressions with difference $m$.  Thus, either $\ell m - 2 \equiv \ell m \mod k$ or $\ell m - 2 \equiv (\ell + 1)m \mod k$.  In the former case, $2 \equiv 0 \mod k$, but this would imply $k = 2$, contradicting the fact that $k \geq 4$.  In the latter case, $m \equiv - 2 \mod k$, but this also leads to a contradiction as $k + 1 \leq m \leq 2k - 3$.  Therefore, $x$ avoids $k$-powers and $k$-anti-powers.
\end{proof}

Note that the above bounds are independent of the alphabet size.  This leads to two questions: does $N_\alpha(k,k)$ depend on the size of the alphabet, and if so, in what way?  Note that $N_\alpha(k,k)$ is nondecreasing as $\alpha$ increases.  The following values of $N_2(k,k)$ were computed by Shallit \cite{Sha}. 
\vspace{-0.1cm}
\begin{align*}
N_2(1,1) = 1 \hspace{.75cm} N_2(2,2) = 2 \hspace{.75cm} N_2(3,3) = 9 \hspace{.75cm} N_2(4,4) = 24 \hspace{.75cm} N_2(5,5) = 55
\end{align*} 
For small $\alpha$ and $k$, $N_\alpha(k,k)$ can be computed by testing all $\alpha$-ary strings of small length with a computer.  In particular, we were able to check that $N_4(3,3) = N_2(3,3) = 9$, which implies $N_\alpha(3,3) = 9$ for all $\alpha \geq 2$.  This follows because a word of length $9$ avoiding $3$-anti-powers would use at most $4$ letters.  We also computed that $N_{11}(4,4) = N_2(4,4) = 24$.  It is straightforward to check that a word of length $24$ avoiding $4$-anti-powers would use at most $11$ letters.  Thus, $N_\alpha(3,3)$ and $N_\alpha(4,4)$ are independent of $\alpha$.  It remains open if this is true for all $k$.

Another scenario to investigate is under what conditions a word can be extended (in a potentially larger alphabet) and still avoid $k$-powers and $k$-anti-powers.  We aim to show that for large enough $\alpha$, no word of length $N_\alpha(k,k) - 1$ can be extended (in a larger alphabet) and avoid $k$-powers and $k$-anti-powers.  To do so, we require the following lemma.

\begin{lem}\label{alphabet mess}
If there exists $\alpha \geq 2$ such that $N_\alpha(k,k) < N_{\alpha + 1}(k,k)$, then one of the following must hold:
\begin{enumerate}
    \item Let $W_{k,\alpha}$ be the set of words on $[\alpha + 1]$ of length $N_\alpha(k,k)$ that avoid $k$-powers and $k$-anti-powers.  For every word $w \in W_{k,\alpha}$, the two factors of $w$ of length $N_\alpha(k,k)-1$ each use exactly $\alpha + 1$ letters.
    \item There exists a word on $[\alpha]$ of the form 
    $$w = u_1(1u_1)^{k-1}x_1 =  u_2(2u_2)^{k-1}x_2 = \cdots =  u_\alpha(\alpha u_\alpha)^{k-1}x_\alpha$$
    that avoids $k$-powers and $k$-anti-powers, where $x_1,\dots,x_\alpha,u_1,\dots,u_\alpha$ are finite words, $|u_1| < \cdots < |u_\alpha|$, and for all $1 \leq i < j \leq \alpha$, 
    $$\gcd(|u_i| + 1, |u_j| + 1) \leq \frac{|u_j| + 1}{k-1}.$$
\end{enumerate}
\end{lem}
\begin{proof}
Suppose that the first case does not hold.  There is a word $w$ of length $N_\alpha(k,k) - 1$ on $[\alpha]$ such that $(\alpha + 1)w$, the extension of $w$ by the addition the letter $\alpha + 1$ on the left, avoids $k$-powers and $k$-anti-powers.  Since $|w|$ is maximal for words on $[\alpha]$ avoiding $k$-powers and $k$-anti-powers, the extension $aw$ contains a $k$-power or $k$-anti-power for any $a \in [\alpha]$.  As $(\alpha + 1)w$ contains no $k$-anti-powers, neither does $aw$ for any $a \in [\alpha]$.  Thus, $aw$ has a prefix that is a $k$-power for each $a \in [\alpha]$.

Hence,
$$w = u_1(1u_1)^{k-1}x_1 =  u_2(2u_2)^{k-1}x_2 = \cdots = u_\alpha (\alpha u_\alpha)^{k-1}x_\alpha,$$
where $x_1,\dots,x_\alpha,u_1,\dots,u_\alpha$ are finite words.  Note that this implies $w[m(|u_\ell| + 1)] = \ell$ for all $1 \leq m \leq k - 1$ and $1 \leq \ell \leq k$.  Without loss of generality (since the labels of the letters are arbitrary), we can assume $|u_1| < \cdots < |u_\alpha|$.  Suppose, seeking a contradiction, that for some $1 \leq i < j \leq \alpha$, we have $\gcd(|u_i| + 1, |u_j| + 1) \geq \frac{|u_j| + 1}{k-1}$.  There is some $1 \leq m \leq k-1$ such that $m(|u_i| + 1) \equiv 0 \mod (|u_j| + 1)$.  Hence, $m(|u_i| + 1) = d(|u_j| + 1)$ for some $1 \leq d \leq k-1$.  However, this implies 
$$i = w[m(|u_i| + 1)] = w[d(|u_j| + 1)] = j.$$
Since we assumed $i< j$, we've reached a contradiction.
\end{proof}

An investigation of the failure of the first case leads to the following corollary.

\begin{cor}\label{no extension}
Suppose $\alpha > \frac{N_\alpha(k,k)}{k} - k + 3$.  If a word $w$ has a factor $u \not= w$ of length $N_\alpha(k,k) - 1$ that uses only $\alpha$ letters, $w$ contains a $k$-power or $k$-anti-power.
\end{cor}
\begin{proof}
Suppose, seeking a contradiction, that $w$ is as above but contains no $k$-power or $k$-anti-power.  For all $1 \leq i < j \leq \alpha$, we have by Lemma \ref{alphabet mess} that
$$1 \leq \gcd(|u_i| + 1, |u_j| + 1) \leq \frac{|u_j| + 1}{k-1}.$$
Thus, $|u_j| \geq k - 2$ for all $j \geq 2$.  Since the $|u_j|$'s are strictly increasing, this implies $|u_\alpha| \geq (\alpha - 2) + (k - 2) = \alpha + k - 4$.  As
$w = u_\alpha(\alpha u_\alpha)^{k-1}x_\alpha$, we have 
$$|w| \geq k(\alpha + k - 4) + k-1 = k\alpha + k^2 - 3k - 1.$$
Since $w$ is a word on $[\alpha]$ avoiding $k$-powers and $k$-anti-powers, $k\alpha + k^2 - 3k - 1 \leq N_\alpha(k,k)$.  If this inequality is not satisfied, then we can conclude $w$ is as above but contains a $k$-power or $k$-anti-power.
\end{proof}

\section{Block Patterns and Their Expectation}\label{expect}
In this section, we return to the general setting of block-patterns to calculate the expected number of $(\mu_1,\dots,\mu_n)$-block-patterns in a word of length $n$ on an alphabet of size $\alpha$.  The special case of this expectation for $k$-powers was calculated by Christodoulakis, Christou, Crochemore, and Iliopoulos in \cite{CCCI}.

\begin{thm}\label{counting powers}\emph{(\cite{CCCI}, Theorem 4.1)}
On average, a word of length $n$ has $\Theta(n)$ $k$-powers.  More precisely, this number is 
$$(n + 1)\frac{\alpha^{1-k}(1 - \alpha^{(1-k)\lfloor\frac{n}{k}\rfloor})}{1 - \alpha^{1-k}} - \frac{k}{\alpha^{k - 1}}\left( \frac{1}{1-\alpha^{1-k}} - \frac{\lfloor \frac{n}{k} \rfloor \alpha^{(1-k)\lfloor\frac{n}{k}\rfloor}}{1 - \alpha^{1-k}} + \frac{\alpha^{1-k}(1 - \alpha^{(1-k)(\lfloor\frac{n}{k}\rfloor - 1)})}{(1 - \alpha^{1-k})^2}\right).$$
\end{thm}

\begin{thm}\label{counting patterns}
On average, a word of length $n$ has $O(n^2)$ and $\Omega(n)$ $(\mu_1,\dots,\mu_k)$-block-patterns. More precisely, the expected number of $(\mu_1,\dots,\mu_k)$-block-patterns is
$$\sum_{m = 1}^{\lfloor \frac{n}{k}\rfloor} (n + 1 -km) \frac{k!}{\mu_1!\cdots\mu_k!} \frac{1}{\alpha^{km}}\prod_{\ell = 1}^{\mu_1+\cdots + \mu_k} \left(\alpha^{m} - (\ell - 1)\right).$$
\end{thm}
\begin{proof}
Let $x$ be a word of length $n$, drawn uniformly at random.  Let 
$$X_{i,j} = \begin{cases} 1 & \text{ if $x[i..j]$ is a $(\mu_1,\dots,\mu_k)$-block-pattern;}\\ 0 & \text{ otherwise.} \end{cases}$$
Let $N = \sum_{i \leq j} X_{i,j}$.  That is, $N$ is the number of $(\mu_1,\dots,\mu_k)$-block-patterns in $x$. We have
\begin{align*}
\E\left[N\right] &= \E\left[\sum_{i = 1}^{n - k + 1} \sum_{j = i + 1}^n X_{i,j}\right]\\
&=\sum_{i = 1}^{n - k + 1} \sum_{j = i + 1}^n \E\left[X_{i,j}\right]\\
&=\sum_{i = 1}^{n - k + 1} \sum_{j = i + 1}^n \P\left(\text{$x[i..j]$ is a $(\mu_1,\dots,\mu_k)$-block-pattern}\right).
\end{align*}

Let us count the number of $(\mu_1,\dots,\mu_k)$-block-patterns of length $\alpha^{j + 1 - i}$ on $[\alpha]$.  Partition $[k]$ into unlabeled parts with $\mu_s$ parts of size $s$, and choose $\mu_1 + \cdots + \mu_k$ distinct ordered elements from $\alpha^{(j + 1 - i)/k}$.  We can assign elements to parts by order of appearance of the parts, which will yield a $(\mu_1,\dots,\mu_k)$-block-pattern.  Moreover, the block-pattern is uniquely determined by the choice of an unlabeled partition and ordered $m$-tuple.  Let $[A]$ denote the indicator function of the event $A$.  We have 
\begin{align*}
\E\left[N\right] &=\sum_{i = 1}^{n - k + 1} \sum_{j = i + 1}^n \frac{k!}{\mu_1!\cdots\mu_k!} \frac{1}{\alpha^{j + 1 - i}}\prod_{\ell = 1}^{\mu_1+\cdots + \mu_k} \left(\alpha^{(j + 1 - i)/k} - (\ell - 1)\right)[j + 1 - i \equiv 0 \mod k]\\
&= \sum_{m = 1}^{\lfloor \frac{n}{k}\rfloor} (n + 1 -km) \frac{k!}{\mu_1!\cdots\mu_k!} \frac{1}{\alpha^{km}}\prod_{\ell = 1}^{\mu_1+\cdots + \mu_k} \left(\alpha^{m} - (\ell - 1)\right).
\end{align*}
Since there are only ${n \choose 2} + n$ nonempty factors of $x$, we have $\E\left[N\right] = O(n^2)$.  Note that the expectation is minimized for $k$-powers, where $\mu_k = 1$ and $\mu_s = 0$ for all $s < k$.  Thus, from Theorem \ref{counting powers}, we have $\E\left[N\right] = \Omega(n)$.
\end{proof}

\begin{cor}
On average, a word of length $n$ has $\Theta(n^2)$ $k$-anti-powers.  More precisely, the expected number of $k$-anti-powers is
$$\sum_{m = 1}^{\lfloor \frac{n}{k}\rfloor} (n + 1 - km) \prod_{\ell = 0}^{k-1} \left(1 - \frac{\ell}{\alpha^m}\right).$$
\end{cor}
\begin{proof}
The formula follows Theorem \ref{counting patterns} in the case $\mu_1 = k$ and $\mu_s = 0$ for $s > 1$. Restricting the sum (of nonnegative terms) to the range $\left\lfloor \frac{n}{4k}\right\rfloor \leq m \leq \left\lfloor \frac{3n}{4k}\right\rfloor$, we see
$$\E\left[\#(\text{$k$-anti-powers in $x$)}\right] \geq \left(\frac{n}{2k} - 1\right)\left(\frac{n}{4} + 1\right) \left( 1 - \frac{k- 1}{\alpha^{\lfloor \frac{n}{4k}\rfloor}}\right)^{k-1}.$$
For $n > 4k\left(1 + \log_\alpha \frac{k - 1}{1 - 2^{-k-1}}\right)$, we have $\left( 1 - \frac{k- 1}{\alpha^{\lfloor \frac{n}{4k}\rfloor}}\right)^{k-1} > \frac{1}{2}$, hence
$$\E\left[\#(\text{$k$-anti-powers in $x$)}\right] > \frac{n^2}{16k} - \frac{n}{4} - \frac{1}{2} = \Omega(n^2).\qedhere$$
\end{proof}

\section{Further Directions}\label{further}

Recall that Theorem \ref{inf powers or patterns} shows that having a small enough density of $(\mu_1,\dots,\mu_k)$-block-pattern prefixes with few equal blocks implies the existence of arbitrarily long power prefixes.  We believe that a strengthening of this argument could yield a lower bound on the density of $P(x,k)$, the set of $m \in \N$ such that the prefix of $x$ of length $km$ is a $k$-power.  In Section \ref{block general}, we also remark that Theorem 6 of \cite{FRSZ}, stating that if $x$ is an $\omega$-power-free word then $AP(x_{(j)},k)$ is nonempty for every $j$ and $k$, is false if we allow infinite alphabets.  Perhaps there is a finer characterization of which $\omega$-power-free words fail this condition.

As in the bounds found by Fici et al. \cite{FRSZ}, our upper and lower bounds for $N_\alpha(k,k)$ are polynomials in $k$ whose degrees differ by $3$.  If it is the case that $N_\alpha(k,k)$ depends on $\alpha$, such a dependence could be used to strengthen the bounds for $N_\alpha(k,k)$.  Given the few known values of $N_\alpha(k,k)$, it seems plausible that $k$ always divides $N_\alpha(k,k)$.  On the other hand, if $N_\alpha(k,k)$ is independent of $\alpha$, this alone would be an interesting structural property of the set of words avoiding $k$-powers and $k$-anti-powers achieving the length $N_\alpha(k,k)$ for arbitrary $\alpha$. We believe the second case holds.
\begin{conj}
The quantity $N_\alpha(k,k)$ is independent of $\alpha$.
\end{conj}

Whether there exist aperiodic recurrent words avoiding $4$ or $5$ powers remains an open question.  One may wish to investigate other large classes of words, such as the morphic words, and their potential to avoid $k$-anti-powers.  A natural generalization is to find the structure of infinite words avoiding $(\mu_1,\dots,\mu_k)$-block-patterns other than $(k,\lambda)$-anti-powers.

Lastly, let $\A = \{a_1,\dots,a_\alpha\}$ be a finite alphabet.  The {\it Parikh vector} $\mathcal{P}(w) = (e_1,\dots,e_\alpha)$ of a finite word $w$ on $\A$ has entry $e_i$ equal to the number of instances of $a_i$ in $w$.  Define an {\it abelian $(\mu_1,\dots,\mu_k)$-block-pattern} to be a word of the form $w = w_1\cdots w_k$ where, if the set $\{1,\dots,k\}$ is partitioned via the rule $i \sim j \iff \mathcal{P}(w_i) =\mathcal{P}(w_j)$, there are $\mu_s$ parts of size $s$ for all $1 \leq s \leq k$.  Ones may ask questions similar to those addressed in this paper for abelian $(\mu_1,\dots,\mu_k)$-block-patterns.

\section{Acknowledgements}
The author would like to thank Joe Gallian for his tireless efforts to foster a productive and engaging mathematical community.  She also extends her gratitude to Colin Defant and Samuel Judge for carefully reading through this article and providing helpful feedback.  This research was conducted at the University of Minnesota, Duluth  REU  and  was  supported  by  NSF/DMS  grant  1650947  and  NSA grant H98230-18-1-0010.

\end{document}